\documentclass[a4paper,11pt,onecolumn]{amsart}
\usepackage{amsfonts}
\usepackage{pb-diagram}
\title[Effective indices of subgroups in B--P groups]{Effective indices of subgroups in Baumslag--Pride groups with free quotients}

\newtheorem{thm}{Theorem}

\newtheorem{cor}[thm]{Corollary}

\numberwithin{thm}{section}

\newcommand{\al}{\alpha}

\newcommand{\bZ}{\mathbb{Z}}

\DeclareMathOperator{\Aut}{Aut}

\author[T. Koberda]{Thomas Koberda}
\address{Department of Mathematics\\ Harvard University\\ 1 Oxford St.\\ Cambridge, MA 02138 }
\email{ koberda@math.harvard.edu}
\begin{document}
\begin{abstract}
Given a group with at least two more generators than relations, we give an effective estimate on the minimal index of a subgroup with a nonabelian free quotient.  We show that the index is bounded by a polynomial in the length of the relator word.  We also provide a lower bound on the index.
\end{abstract}
\maketitle
\begin{center}
\today
\end{center}
\section{Introduction and Statement of Results}
Let $G$ be a finitely presented group.  We say that $G$ is a {\bf Baumslag-Pride group} or {\bf BP--group} if it admits a presentation $G=\langle S\mid R\rangle$, with $|S|\geq |R|+2$.  In particular, if $G$ is a one-relator group with at least three generators, then $G$ is a BP--group.  The terminology stems from a paper of those two authors, where they prove that every such group contains a finite index subgroup which admits a surjection onto a nonabelian free group, i.e. $G$ is {\bf large} (see \cite{BP}).

In their proof, Baumslag and Pride explicitly produce the subgroup $H$.  They choose a special presentation for $G$ in which they assume that a particular generator appears with zero exponent sum in all the relator words.  In the sequel, we shall call such presentations {\bf good}.  When $G$ is a BP--group, it is always possible to find a good presentation.  Once such a presentation is found, it is possible to produce $H$ in such a way so that $[G:H]$ is no more than linear in the length of the longest relator word.

In general, one will not be so lucky as to be given a good presentation.  Given a relator word $w\in F_n$, there is always an automorphism $\al\in \Aut(F_n)$ such that $\al(w)$ has zero exponent sum in some generator, but the word length $\ell(\al(w))$ might be somewhat longer than $\ell(w)$.  The main result of this note is:

\begin{thm}\label{t:main}
Let $w\in F_n$, $n\geq 2$, and fix a free generating set for $F_n$.  Then there is a polynomial $p_n$ depending only on $n$ and an $\al\in \Aut(F_n)$ such $\al(w)$ has zero exponent sum in at least one of the generators and such that $\ell(\al(w))\leq p_n(\ell(w))$.
\end{thm}

With a little bit of work, we obtain:

\begin{cor}\label{c:length}
There is an automorphism of the free group $\al$ and a polynomial $p$ depending only on the rank of $G$ and the number of relators of $G$ such that $\al(w)$ has zero exponent sum in one fixed generator for all relator words $w$, and such that $\ell(\al(w))\leq p(\ell(w))$ for all such $w$.
In particular, suppose $G$ is a BP--group with relator set $R$.  Then there is a polynomial $p$ depending only on the rank of $G$ and the number of relators of $G$, and a subgroup $H<G$ such that $H$ admits a surjection to $F_2$ and \[[G:H]\leq \max_{w\in R} p(\ell(w)).\]
\end{cor}

We remark that the polynomial can be chosen universally, which is to say independently of $n$.  We shall show in the proof of Theorem \ref{t:main} that there is a polynomial which works for $F_2$, and hence for all finite rank free groups.  The smallest degree that works may decrease as the rank gets large.

It might be guessed that every BP--group already surjects onto a nonabelian free group, in which case Corollary \ref{c:length} has no content.  However, we will prove:

\begin{thm}\label{t:lower}
Let $N$ be fixed.  Then there is a word $w\in F_n$ of length at most $(5N)!$ such that no subgroup of $F_n/\langle w\rangle$ of index at most $N$ admits a surjection to $F_2$.
\end{thm}

This bound makes effective an example of R. Lyndon which appears on pages 114--115 in J. Stallings' article \cite{St}.  In that example, Lyndon produces for each $n$ and each sequence of $(n-1)n/2$ distinct powers of $2$ a one--relator group which does not surject onto a nonabelian free group.

To prove Theorem \ref{t:lower}, we will need the following result which can be found in \cite{A}:
\begin{thm}
Let $G$ be a finite group.  Then for all $n$ there exists a word $w\in F_n$ such that for all $g_1,\ldots, g_n$ satisfies $w$ if and only if the subgroup $\langle g_1,\ldots, g_n\rangle$ is solvable.
\end{thm}

The lower bound will follow easily from this result.

\section{Acknowledgements}
The author wishes to thank M. Ab\'ert, under whose auspices most of this research was done while the author was an undergraduate.  The author also thanks J. Huizenga, S. Isaacson, and A. Silberstein for critically reading some early manuscripts, and T. Church for several useful comments.  The author also thanks the referee for comments which led to the simplification of many of the arguments and exposition.  The author is partially supported by an NSF Graduate Research Fellowship.

\section{Automorphism orbits of words and the proofs of the results}
We first recall the well-known fact about a generating set for $\Aut(F_n)$: $\Aut(F_n)$ is finitely generated by so-called elementary Nielsen transformations (see \cite{LS} for more details).  If $X$ is a free generating set for $F_n$, these amount to replacing some $x\in X$ with $x^{-1}$, or for distinct $x,y\in X$, replacing $x$ by $x\cdot y$.  Though it is not standard, we include $x\mapsto x\cdot y^n$ for each $n\in \bZ$ in the definition of elementary Nielsen transformations.  It is evident that the application of elementary Nielsen transformations to a word $w$ runs the Euclidean algorithm on the vector $(X_1(w),\ldots,X_n(w))$, where $X_i(w)$ is the exponent sum of the generator $x_i$ in $w$.

\begin{proof}[Proof of Theorem \ref{t:main}]
Clearly it suffices to prove the statement for $F_2$.  Let $w\in F_2$ be fixed.  Suppose that $x$ and $y$ are generators with exponent sums $X(w)$ and $Y(w)$ in $w$ respectively.  Clearly $|X(w)|+|Y(w)|\leq \ell(w)$.  It is standard that there are universal constants $C$ and $D$ such that the algorithm will terminate in $C\log(\ell(w))+D$ steps.

Let $L_x(w)$ and $L_y(w)$ denote the number of occurrences of $x$ and $y$ in $w$.  We have that \[\ell(w)=L_x(w)+L_y(w)+L_{x^{-1}}(w)+L_{y^{-1}}(w).\]  Replacing generators by their inverses if necessary, we may assume that $X(w)$ and $Y(w)$ are positive.  Suppose that $X(w)\leq Y(w)$.  We may choose an $c_1$ such that $c_1\cdot X(w)\leq Y(w)$ but $(c_1+1)\cdot X(w)>Y(w)$.  Replacing $x$ by $xy^{-c_1}$ will result in a new word $w'$ which satisfies $Y(w')=Y(w)-c_1\cdot X(w)$, $L_y(w')\leq L_y(w)+c_1\cdot L_x(w)$, and $L_x(w')\leq L_x(w)$.  It follows that \[\ell(w')\leq (c_1+1)\ell(w).\]

Note that if $Y(w)\geq X(w)$ then $X(w')\geq Y(w')$.  Applying the algorithm $n$ times will result in a word $w^{(n)}$, and suppose that \[\ell(w^{(n)})\leq (c_n+1)\cdot \ell(w^{(n-1)}),\] so that \[\ell(w^{(n)})\leq\left(\prod_{i=1}^n (c_i+1)\right)\cdot \ell(w).\]  Note that $c_i\cdot X(w^{(i)})\leq Y(w^{(i)})$ or $c_i\cdot Y(w^{(i)})\leq X(w^{(i)})$.  Note that $c_1\leq \ell(w)/2$ and \[\max\{X(w'),Y(w')\}\leq \ell(w)/2.\]  We suppose inductively that \[\max\{X(w^{(i)}),Y(w^{(i)})\}\leq \ell(w)/2^{i/2}.\]  We also suppose inductively that $c_i\leq \max\{\ell(w)/2^{i/2},1\}$.  It is possible that at any step, the algorithm will terminate because $X(w^{(i)})=Y(w^{(i)})$.

Suppose $(1/2)Y(w^{(i)})\leq X(w^{(i)})\leq Y(w^{(i)})$.  Then $c_{i+1}=1$, and \[\max\{X(w^{(i+1)}),Y(w^{(i+1)})\}\leq \ell(w)/2^{i/2}.\]  Note that since $Y(w^{(i)})-X(w^{(i)})\leq (1/2)Y(w^{(i)})$, we will have \[\max\{X(w^{(i+2)}),Y(w^{(i+2)})\}\leq \ell(w)/2^{(i/2+1)}.\] Otherwise, we may assume $X(w^{(i)})<(1/2)Y(w^{(i)})$.  If $X(w^{(i)})\neq 1$ we will have \[\max\{X(w^{(i+1)}),Y(w^{(i+1)})\}\leq \ell(w)/2^{(i/2+1)},\] and $c_{i+1}\leq (1/2)Y(w^{(i)})$.  If $X(w^{(i)})=1$ then the algorithm will terminate.

We can now estimate the quantity \[\prod_{i=1}^n (c_i+1).\]  Let $M=M(w)$ be the least integer greater than $C\log(\ell(w))+D$.  We have that $n\leq M$.  At each step of the algorithm, we have that $c_i+1\geq 2$.  Also it is possible that $c_i=\max\{X(w^{(i)}),Y(w^{(i)})\}$, but this can only happen once.  Otherwise, we have argued by induction that $c_i$ decays like $\ell(w)/2^{i/2}$.  We may thus estimate \[P=\prod_{i=1}^n (c_i+1)\leq \ell(w)\cdot 2^M\cdot\prod_{i=1}^M\left(\frac{\ell(w)}{2^{i/2}}+1\right).\]  The first factor is from the possibility that $c_i=\max\{X(w^{(i)}),Y(w^{(i)})\}$, which we may assume to be no more than $\ell(w)-1$.  The second comes from the possibility of $c_i=1$ at any step, and the third is the estimate for $|c_i+1|$ in the remaining possible case.

Note that the third factor can be rewritten as \[\prod_{i=1}^M\left(\frac{\ell(w)+2^{i/2}}{2^{i/2}}\right).\]  At every step of the algorithm we assume that $1\leq c_i\leq \ell(w)/2^{i/2}$.  In particular, we may estimate the third term to be dominated by \[\prod_{i=1}^M\left(\frac{2\ell(w)}{2^{i/2}}\right).\]

For compactness of notation, write $x$ for $\ell(w)$.  Notice that $M=M(x)$ depends on $x$, and varies like $\log x$.  We take the logarithm of the estimate on $P$.  We obtain the expression \[\log x+2M\log 2+M\log x-\frac{(M^2+M)}{4}\cdot\log 2.\]  For $x$ sufficiently large, we may replace $M$ by a constant $N$ times $\log x$.  Rewriting, we get \[\log x+2N\log 2\log x+N(\log x)^2-(\log 2)\cdot\frac{N^2(\log x)^2+N\log x}{4}.\]

We can replace $N$ by any sufficiently large constant.  We may therefore suppose that the coefficient of $(\log x)^2$ is negative.  Let $K$ be large enough so that \[\log x+2N\log 2\log x-\frac{\log 2}{4}\cdot N\log x -K\log x\] is negative for all $x$ sufficiently large.  It follows that if $0\ll\ell(w)$, \[P\cdot\ell(w)^{-K}\leq 1,\] the desired conclusion.
\end{proof}

To establish Corollary \ref{c:length}, we note the following observation from the proof of the main result of \cite{BP}: if $G=\langle S\mid R\rangle$ is a BP--group with a given good presentation then $G$ has a finite index subgroup $H$ which surjects to $F_2$, and which has linear index in the length of the relators of $G$.

\begin{proof}[Proofs of Corollary \ref{c:length}]
In order to achieve their main theorem, Baumslag and Pride assume that some generator appears with zero exponent sum in each one of the relators.  One way to produce such a presentation of $G$ is as follows: choose one relator and apply automorphisms so that the exponent sum in all the generators is zero except for at most one.  This is possible, since we may order the generators as $x_1,\ldots,x_n$ with exponent sums $X_1(r_1),\ldots,X_n(r_1)$ in the first relator word $r_1$, and run the Euclidean algorithm on successive pairs of exponent sums.  The exponent sum $X_n(r_1)$ cannot be generally eliminated unless we are very lucky.

We then take the next relator $r_2$ and eliminate $X_1(r_2)$.  This can be done by applying automorphisms that run the Euclidean algorithm on $X_1(r_2)$ and $X_2(r_2)$.  This way, $X_1(r_1)$ will remain unchanged.  We repeat this procedure in order to eliminate $X_1(r_2),\ldots,X_{n-2}(r_2)$.  Repeating this procedure, we can produce a good presentation on a general BP--group.  Note that it is in fact essential that $G$ be a BP--group: indeed, for the $k^{th}$ relator $r_k$, we can only eliminate $X_1(r_k),\ldots,X_{n-k}(r_k)$.  We finally note that composing two polynomials results in a polynomial, so that we may apply Theorem \ref{t:main}.
\end{proof}

\begin{proof}[Proof of Theorem \ref{t:lower}]
In \cite{A}, Ab\'ert works out the case of $G=S_{5N}$ in detail.  He shows that for each $n$ and $N$ there is a word $w\in F_n$ such that no subgroup of index at most $N$ of $F_n/\langle w\rangle$ admits a surjection to $F_2$.  Furthermore, the length of $w$ is no longer than the longest word in $S_{5N}$ with respect to any pair of generators, and thus has length at most $(5N)!$.
\end{proof}

\end{document}